\documentclass[draft,a4]{amsart}
\usepackage{amsmath,amssymb,amsthm}
\usepackage{ascmac}
\usepackage[nobysame]{amsrefs}
\usepackage{amscd}
\usepackage[margin=28mm]{geometry}
\usepackage{enumerate}
\usepackage{graphics}

\usepackage{color}
\usepackage{ulem}

\newtheorem{Thm}{Theorem}[section]
\newtheorem{Pro}[Thm]{Proposition}
\newtheorem{Lem}[Thm]{Lemma}

\theoremstyle{definition}
\newtheorem{Def}[Thm]{Definition}

\newtheorem{Rem}{Remark}

\newcommand{\INN}[2]{\langle #1, #2 \rangle}

\title[Special Lagrangian submanifolds and cohomogeneity one actions on $\mathbb{C}P^{n}$]{Special Lagrangian submanifolds and cohomogeneity one actions on the complex projective space}
\date{\today}

\author{Masato Arai}
\address{Faculty of Science, Yamagata University, Yamagata 990-8560, Japan}
\email{arai@sci.kj.yamagata-u.ac.jp}

\author{Kurando Baba}
\address{Department of Mathematics, Faculty of Science and Technology, Tokyo University of Science, Noda, Chiba 278-8510, Japan}
\email{baba$\_$kurando@ma.noda.tus.ac.jp}

\keywords{special Lagrangian submanifold, Calabi-Yau manifold, calibration, cohomogeneity one action}
\subjclass[2010]{Primary~53C38}

\begin{document}

\maketitle

\begin{abstract}
We construct 
examples of cohomogeneity one special Lagrangian submanifolds
in the cotangent bundle over the complex projective space,
whose Calabi-Yau structure was given by Stenzel.
For each example,
we describe the condition of special Lagrangian
as an ordinary differential equation.
Our method is based on a moment map technique and
the classification of cohomogeneity one actions
on the complex projective space classified by Takagi.
\end{abstract}

\section*{Introduction}
Calibrated submanifolds, which are a special class of minimal submanifolds, were firstly introduced by Harvey and Lawson (\cite{HaLa}). Their study has led to extensive research on calibrated geometry both in mathematics and physics. One of most interesting  calibrated submanifolds are special Lagrangian submanifolds in Calabi-Yau manifolds since they play an important role in understanding of mirror symmetry in string theory (\cite{SYZ}). It is expected that explicit construction of special Lagrangian submanifolds promote precise understanding of mirror symmetry and related problems.

There have been a lot of progress of constructing special Lagrangian submanifolds since Harvey and Lawson introduced them. 
One of useful techniques to construct such submanifolds is a moment map, which was introduced by Joyce (\cite{Jo}).
In order to use the moment map, one needs to impose large symmetry groups on special Lagrangian manifolds. The advantage of this technique is that the condition for special Lagrangian submanifolds is reduced to an ordinary differential equation which is explicitly solved. With the use of this technique, Joyce constructed special Lagrangian submanifolds in $\mathbb{C}^{n}(\cong T^{*}\mathbb{R}^{n})$ invariant under a subgroup of $SU(m)$ which is cohomogeneity one symmetry group. 
The Joyce's method was applied to construction of special Lagrangian submanifolds of Calabi-Yau manifolds on the cotangent bundles over rank one compact Riemannian symmetric spaces, given by Stenzel (\cite{St}, \cite{St93}). Explicit examples of such a class, especially cohomogeneity one special Lagrangian manifolds, were studied in
$T^{*}S^{n}$ over the sphere $S^{n}$ by Anciaux (\cite{An}),
Ionel and Min-Oo (\cite{IoMO}), and Hashimoto and Sakai (\cite{HS}).
Another method, called a bundle technique introduced by Harvey and Lawson (\cite{HaLa}), was also considered to construct special Lagrangian submanifolds of $T^{*}S^{n}$ by Karigiannis and Min-Oo (\cite{KaMO}) and $T^{*}\mathbb{C}P^{n}$ by Ionel and Ivey (\cite{IoIv}). While special Lagrangian submanifolds of $T^{*}S^{n}$ were studied by both the moment map technique and the bundle technique, ones of $T^{*}\mathbb{C}P^{n}$ were considered only by the latter.

In this paper, we construct cohomogeneity one special Lagrangian submanifolds
in $T^{*}\mathbb{C}P^{n}$ over the complex projective space $\mathbb{C}P^{n}$
by using the moment map technique. For this purpose we use the classification of 
cohomogeneity one actions on $\mathbb{C}P^{n}$ which was given by Takagi (\cite{T}).
The moment map technique with the above classification allows us to construct nontrivial special 
Lagrangian submanifolds which are realized as extension of homogeneous hypersurfaces 
in $\mathbb{C}P^{n}$ to $n$-dimensional submanifolds in $T^{*}\mathbb{C}P^{n}$.

The organization of this paper is as follows.
In Section \ref{sec:preliminaries},
we review
the notions of Calabi-Yau manifolds and special Lagrangian submanifolds,
and describe the Calabi-Yau structure on
$T^{*}\mathbb{C}P^{n}$ given by Stenzel (\cite{St}).
We explain Hashimoto-Sakai's method (\cite{HS}) in which
Joyce's moment map technique is applied to construct cohomogeneity 
one special Lagrangian submanifolds in $T^{*}S^{n}$.
We also explain the geometry of cohomogeneity one actions
on $\mathbb{C}P^{n}$ given by Takagi (\cite{T}) and some related notions.
In Section \ref{sec:constSLag},
we apply Hashimoto-Sakai's method to the case of $T^{*}\mathbb{C}P^{n}$.
Finally we construct cohomogeneity one special Lagrangian submanifolds
in $T^{*}\mathbb{C}P^{n}$
(Theorems \ref{thm:slagAIIIAIII}--\ref{thm:slagDIII}).

\section{Preliminaries}\label{sec:preliminaries}
\subsection{The Calabi-Yau structure on $T^{*}\mathbb{C}P^{n}$}

We review the notion of Calabi-Yau manifold (cf.~\cite{Jotext}).

\begin{Def}
An \textit{almost Calabi-Yau manifold} is a quadruple $(M,J,\omega,\Omega)$
such that $(M,J,\omega)$ is a K\"ahler manifold of complex dimension $n(\geq 2)$ with a complex structure $J$ and a K\"aher 2-form $\omega$,
and a non-vanishing holomorphic $(n,0)$-form $\Omega$ on $M$.
In addition,
if $\omega$ and $\Omega$ satisfy the following condition,
then $(M,J,\omega,\Omega)$ is called a \textit{Calabi-Yau manifold}:
\begin{equation}\label{eqn:cycondi}
\dfrac{\omega^{n}}{n!}=(-1)^{n(n-1)/2}\left(\dfrac{i}{2}\right)^{n}\Omega\wedge\bar{\Omega}.
\end{equation}
\end{Def}

\noindent
A metric $g$ on a Calabi-Yau manifold $(M,J,\omega,\Omega)$ is 
constructed from $J$ and $\omega$
by using the equation $\omega(v,w)=g(Jv,w)$ 
for all vector fields $v, w$ on $M$.
The simplest example of a Calabi-Yau manifold is a
complex Euclidean space
$\mathbb{C}^{n}$ equipped with 
the flat metric $g_0$, the K\"ahler form $\omega_0$ and the holomorphic volume form $\Omega_0$.
For complex coordinates $(z_{1}, \dotsc, z_{n})$, they are given as
\begin{equation*}
g_{0}=|dz_{1}|^{2}+\dotsm+|dz_{n}|^{2},\quad
\omega_{0}=\dfrac{i}{2}(dz_{1}\wedge d\bar{z}_{1}+\dotsm+dz_{n}\wedge d\bar{z}_{n}),\quad
\Omega_{0}=dz_{1}\wedge\dotsm\wedge dz_{n}.
\end{equation*}

In this paper, we consider a nontrivial Calabi-Yau manifold, 
the cotangent bundle $T^{*}\mathbb{C}P^{n}$ over
the $n$-dimensional complex projective space $\mathbb{C}P^{n}$.
The base manifold belongs to a class of rank one compact Riemannian symmetric spaces.
On the cotangent bundles over these spaces Ricci-flat K\"ahler metrics were constructed (\cite{St}, \cite{St93}), 
which are called the \textit{Stenzel metric}.
We utilize a convenient description of the Stenzel metric of $T^*\mathbb{C}P^{n}$
with local coordinates on $\mathbb{C}P^{n}\times\mathbb{C}P^{n}$ (\cite{Le}). 
For this purpose, we choose the base manifold as
the standard Fubini-Study metric on $\mathbb{C}P^{n}$
and identify $T^{*}\mathbb{C}P^{n}$
with the tangent bundle $T\mathbb{C}P^{n}$.

Let us start with $B=\{(\zeta,\xi)\in\mathbb{C}^{n+1}\times\mathbb{C}^{n+1}\mid\zeta\neq0,\xi\cdot\bar{\zeta}=0\}$,
where the dot product $\cdot$ is $\mathbb{C}$-bilinear:
$\xi \cdot \bar{\zeta}=\xi_{0}\bar{\zeta}_{0}+ \dotsb +\xi_{n}\bar{\zeta}_{n}$ for $\zeta=(\zeta_{0}, \dotsc, \zeta_{n})$ and
$\xi=(\xi_{0}, \dotsc, \xi_{n})$.
From this we have $T\mathbb{C}P^{n}\cong B/\mathbb{C}^{*}$ (\cite{IoIv}), 
where $\mathbb{C}^{*}:=\mathbb{C}-\{0\}$
and the $\mathbb{C}^{*}$-action on $B$
is defined by $(\zeta,\xi)\mapsto (\lambda\zeta,\lambda\xi)$
for $(\zeta,\xi)\in B$ and $\lambda\in\mathbb{C}^{*}$.

To describe $T\mathbb{C}P^{n}$ with local coordinate on $\mathbb{C}P^{n}\times\mathbb{C}P^{n}$,
we define the following mapping $\hat{\Phi}$:
\begin{equation*}
\hat{\Phi}:B\to\mathbb{C}^{n+1}\times\mathbb{C}^{n+1};
(\zeta,\xi)\mapsto 
((\cosh\mu)\zeta+i\dfrac{\sinh\mu}{\mu}\xi;(\cosh\mu)\bar{\zeta}+i\dfrac{\sinh\mu}{\mu}\bar{\xi}),\quad\mu:=\dfrac{|\xi|}{|\zeta|}.
\end{equation*}
We shall write elements in $\mathbb{C}^{n+1}\times\mathbb{C}^{n+1}$
as an ordered pair of row vectors in $\mathbb{C}^{n+1}$
separated by a semicolon.
It is then seen that, taking account of $\mathbb{C}^{*}$-action on $B$ and each factor
in $\mathbb{C}^{n+1}\times\mathbb{C}^{n+1}$,
this mapping induces an embedding $\Phi:T\mathbb{C}P^{n} \to \mathbb{C}P^{n}\times\mathbb{C}P^{n}$ (see \cite{Sz}).
By using this map,
we define a manifold $\mathcal{M}=\Phi(T\mathbb{C}P^{n}),$
in which its complex structure induces a complex structure $J_{Stz}$ on $T\mathbb{C}P^{n}$ via $\Phi$.
The mapping $\Phi$ allows us to describe the K\"ahler potential for the Stenzel metric on $T\mathbb{C}P^{n}$ by means of local coordinates on $\mathcal{M}\subset\mathbb{C}P^{n}\times\mathbb{C}P^{n}$  (\cite{Le}):
\begin{equation*}
\mathcal{A}(z;w)=\sum_{j,k=0}^{n}|z_{j}w_{k}|^{2},\quad
\mathcal{B}(z;w)=\sum_{j=0}^{n}z_{j}w_{j}\neq 0,\quad
\mathcal{N}(z;w)=\dfrac{\mathcal{A}(z;w)}{|\mathcal{B}(z;w)|^{2}},
\end{equation*}
where $(z;w)=(z_{0},\ldots,z_{n};w_{0},\ldots,w_{n})$ be the homogeneous
coordinates on $\mathbb{C}P^{n}\times\mathbb{C}P^{n}$.
If $f$ is a solution of
$(2\mathcal{N})\mathcal{N}^{n-1}(f')^{2n}+
2(\mathcal{N}-1)\mathcal{N}^{n}(f')^{2n}f''=1,$
then $\varrho_{Stz}:=f(\mathcal{N})$ gives
the K\"ahler potential for the Stenzel metric.
Then $\varrho_{Stz}$ is invariant under the action
of $SU(n+1)$ defined by
$(z;w)\mapsto (gz;\bar{g}w)$
for $g\in SU(n+1)$ and $(z;w)\in\mathcal{M}(\subset \mathbb{C}P^{n}\times\mathbb{C}P^{n})$.

The K\"ahler potential $f(\mathcal{N})$ leads to
the Liouville $1$-form $\alpha_{Stz}(:=\mathrm{Im}(\bar{\partial}f(\mathcal{N})))$: 
\begin{equation}\label{eqn:alphastz}
\begin{split}
\alpha_{Stz} 
=\dfrac{f'(\mathcal{N})}{2i}\sum_{j=0}^{n}\Bigg[\left\{
\dfrac{z_{j}|w|^{2}}{|\mathcal{B}|^{2}}-
\dfrac{\mathcal{A}}{|\mathcal{B}|^{2}}\dfrac{\bar{w}_{j}}{\bar{\mathcal{B}}}
\right\}&d\bar{z}_{j}+
\left\{
\dfrac{|z|^{2}w_{j}}{|\mathcal{B}|^{2}}-
\dfrac{\mathcal{A}}{|\mathcal{B}|^{2}}\dfrac{\bar{z}_{j}}{\bar{\mathcal{B}}}
\right\}d\bar{w}_{j}\\
&-
\left\{
\dfrac{\bar{z}_{j}|w|^{2}}{|\mathcal{B}|^{2}}-
\dfrac{\mathcal{A}}{|\mathcal{B}|^{2}}\dfrac{w_{j}}{\mathcal{B}}
\right\}dz_{j}-
\left\{
\dfrac{|z|^{2}\bar{w}_{j}}{|\mathcal{B}|^{2}}-
\dfrac{\mathcal{A}}{|\mathcal{B}|^{2}}\dfrac{z_{j}}{\mathcal{B}}
\right\}dw_{j}
\Bigg].
\end{split}
\end{equation}
Note that the K\"ahler $2$-form $\omega_{Stz}$
satisfies $\omega_{Stz}=-d\alpha_{Stz}$.
A simple calculation by means of (\ref{eqn:alphastz})
shows that $\mathcal{L}_{X^{*}}\alpha_{Stz}=0$ holds
for $X\in \mathfrak{su}(n+1)$,
where $\mathcal{L}_{X^{*}}$
denotes the Lie derivative with respect to
the fundamental vector filed $X^{*}$ of $X\in\mathfrak{su}(n+1)$ defined by
$X^{*}_{(z;w)}=(d/dt)|_{t=0}(\exp(t X)z;\overline{\exp(t X)}w)$
for $(z;w)\in\mathcal{M}$.

The above results are readily described in terms of
the inhomogeneous coordinates
\begin{equation*}
(\tilde{z};\tilde{w})=(\tilde{z}_{1},\ldots,\tilde{z}_{n};
\tilde{w}_{1},\ldots,\tilde{w}_{n}),\quad
\tilde{z}_{i}=\dfrac{z_{i}}{z_{0}},\quad
\tilde{w}_{i}=\dfrac{w_{i}}{w_{0}},
\end{equation*}
which leads to
$\mathcal{A}=(1+|\tilde{z}|^{2})(1+|\tilde{w}|^{2})$
and $\mathcal{B}=1+\tilde{z}\cdot\tilde{w}$.
Since
we obtain
$\det\partial\bar{\partial}f(\mathcal{N})=(1/|\mathcal{B}|^{2})^{n+1}$
(cf.  \cite[p.~320]{Le}),
the holomorphic $(n,0)$-form
$\Omega_{Stz}$ defined by
\begin{equation}\label{eqn:omegastz}
\Omega_{Stz}=\dfrac{1}{\mathcal{B}^{n+1}}d\tilde{z}_{1}\wedge\dotsm\wedge d\tilde{z}_{n}
\end{equation}
satisfies the condition (\ref{eqn:cycondi}).
Hence $(T\mathbb{C}P^{n}\cong\mathcal{M},J_{Stz},\omega_{Stz},\Omega_{Stz})$
is a Calabi-Yau manifold.

\begin{Rem}
Stenzel gave a Calabi-Yau structure
$(J_{Stz}',\omega_{Stz}',\Omega_{Stz}')$
of $T^{*}S^{n}$ as follows:
\begin{itemize}
\item The complex structure $J_{Stz}'$ on $T^{*}S^{n}$,
which is identified with 
$TS^{n}(=\{(x,\xi)\in\mathbb{R}^{n+1}\times\mathbb{R}^{n+1}\mid |x|=1,\INN{x}{\xi}=0\})$,
is induced from $Q^{n}(:=\{(z_{0},\dotsb,z_{n})\in\mathbb{C}^{n+1}\mid
z_{0}^{2}+\dotsb+z_{n}^{n}=1\})$ via the following diffeomorphism $\Phi'$
(see \cite{Sz}):
\begin{equation*}
\Phi':TS^{n}\to Q^{n};
(x,\xi)\mapsto\cosh(|\xi|)x+i\dfrac{\sinh(|\xi|)}{|\xi|}\xi.
\end{equation*}
\item The K\"ahler 2-form
$\omega_{Stz}'=(i/2)\partial\bar{\partial}\varrho_{Stz}'$,
where $\varrho_{Stz}'=f(r^{2})$,
is the K\"ahler potential
with $r^{2}=\sum_{j=0}^{n}|z_{j}|^{2}$
and a solution $f$ of
$r(f')^{n}+(r^{2}-1)(f')^{n-1}f''=1$.
\item The holomorphic $(n,0)$-form $\Omega_{Stz}'$ is defined by
\begin{equation*}
\Omega_{Stz}'(v_{1},\ldots,v_{n})=(dz_{0}\wedge dz_{1}\wedge \dotsm\wedge dz_{n})
(Z,v_{1},\ldots,v_{n}),
\end{equation*}
where $Z=z_{0}(\partial/\partial z_{0})+z_{1}(\partial/\partial z_{1})+\dotsb+z_{n}(\partial/\partial z_{n})$.
\end{itemize}
\end{Rem}

\subsection{Special Lagrangian submanifolds in a Calabi-Yau manifold}

We 
explain
the notion of special Lagrangian submanifold introduced by Harvey-Lawson (\cite{HaLa}).

\begin{Def}
Let $(M,J,\omega,\Omega)$ be a Calabi-Yau manifold of complex dimension $n$.
A submanifold $L$ of real dimension $n$
in $M$ is called a \textit{special Lagrangian submanifold} with phase $\psi(\in\mathbb{R})$ if
it is calibrated by the calibration $\mathrm{Re}(e^{i\psi}\Omega)$.
\end{Def}

\noindent
Special Lagrangian submanifolds
are homologically volume-minimizing in Calabi-Yau manifolds (see \cite{HaLa}).
The following result is useful to check
the condition of special Lagrangian
in our argument.

\begin{Pro}[Corollary III.1.11 in \cite{HaLa}]\label{pro:HLpro}
$L$ is a special Lagrangian submanifold with phase $\psi$
in a Calabi-Yau manifold $(M,J,\omega,\Omega)$
if and only if $L$ is a real $n$-dimensional submanifold satisfying two conditions {\rm (i)} $\omega|_{L}\equiv0$
and {\rm (ii)} $\mathrm{Im}(e^{i\psi}\Omega)|_{L}\equiv0$.
\end{Pro}

\noindent
Note that the the condition (i) means that
$L$ is a Lagrangian submanifold in $M$.

\subsection{Review on Hashimoto-Sakai's method}\label{subsec:HSmethod}

A purpose of this paper is to construct
special Lagrangian submanifolds
with the use of a moment map technique.
In particular, we use a method developed by
Hashimoto and Sakai (\cite{HS}).
They constructed special Lagrangian submanifolds
in the cotangent bundle $T^{*}S^{n}$ over the $n$-sphere $S^{n}$
by using cohomogeneity one actions on $S^{n}$,
which is an application
of Joyce's method proposed in \cite{Jo}.
In this subsection, we review on the Hashimoto-Sakai's method.

Let $(M,\omega)$
be a symplectic manifold
and $K$ be a connected Lie group
with Lie algebra $\mathfrak{k}$.
The dual space of $\mathfrak{k}$ is denoted by $\mathfrak{k}^{*}$ and
the pairing of $\mathfrak{k}$ and $\mathfrak{k}^{*}$
is given by $\INN{}{}$.
Let $\rho$ be a Lie homomorphism from $K$ to the diffeomorphism group of $M$.
Then we obtain a $K$-action on $M$ by $(k,x)\mapsto \rho(k)x$
for $k\in K$ and $x\in M$.
The fundamental vector field $X^{*}$ of $X\in \mathfrak{k}$ on $M$
is expressed by
\begin{equation*}
X^{*}_{x}=\dfrac{d}{dt}\biggm|_{t=0}\rho(\exp(t Z))x
=\rho_{*}(X)x
\quad
(x\in M),
\end{equation*}
where $\rho_{*}$ denotes the differential of $\rho$
at the identity element in $K$.
It follows from Cartan's magic formula
that,
if $\omega$
is invariant under the $K$-action,
then $\iota(X^{*})\omega$
is a closed 1-form on $M$ for $X\in \mathfrak{k}$,
where $\iota$ denotes the interior product on $M$.
The $K$-action on $M$
is called \textit{Hamiltonian}
if the 1-form $\iota(X^{*})\omega$ on $M$ is exact for $X\in \mathfrak{k}$.
Suppose that the $K$-action on $M$ is Hamiltonian.
Then there exists a $K$-equivariant map
$\mu:M \to \mathfrak{k}^{*}$,
which is called a \textit{moment map},
such that $\iota(X^{*})\omega=d\mu_{X}$ holds
for $X\in \mathfrak{k}$,
where $\mu_{X}\in C^{\infty}(M)$ is defined by $\mu_{X}(x)=\INN{\mu(x)}{X}$
for $x\in M$.

Let $\mu$ be a moment map for the $K$-action on $M$,
and $\mu^{-1}(c)$ denotes the inverse image of $\mu$ at $c\in \mathfrak{k}^{*}$.  
Set $Z(\mathfrak{k}^{*})=\{X\in\mathfrak{k}^{*}\mid
\mathrm{Ad}^{*}(k)X=X\,(\forall k\in K)\}$, where $\mathrm{Ad}^{*}$ is the coadjoint
representation of $K$ on $\mathfrak{k}^{*}$.
\begin{Lem}
For any $c\in \mathfrak{k}^{*}$,
the inverse image $\mu^{-1}(c)(\subset M)$
is $K$-invariant if and only if
$c\in Z(\mathfrak{k}^{*})$.
\end{Lem}
\noindent
The Hashimoto-Sakai's method for
cohomogeneity one
construction of (special) Lagrangian submanifolds
is based on the following two facts.
\begin{Pro}[Proposition 2.5 in \cite{HS}]\label{pro:HSpro1}
Let $L$ be a
connected $K$-invariant submanifold.
If $L$ is an isotropic submanifold $($i.e.,
$\omega|_{L}\equiv 0$$)$,
then there exists $c\in Z(\mathfrak{k}^{*})$
such that $L$ is contained in $\mu^{-1}(c)$.
\end{Pro}
\begin{Pro}[Proposition 2.6 in \cite{HS}]\label{pro:HSpro2}
Let $L$ be a connected
$K$-invariant
submanifold in $M$.
Assume that the action of $K$ on $L$
is of cohomogeneity one.
Then $L$ is an isotropic submanifold 
if and only if
$L \subset \mu^{-1}(c)$
for some $c\in Z(\mathfrak{k}^{*})$.
\end{Pro}
\begin{Rem}
For the Lie subgroup $K=SO(p)\times SO(n+1-p)$ of $SO(n+1)$,
Hashimoto and Sakai (\cite{HS})
applied the above method to the natural $K$-action
on $T^{*}S^{n}\cong Q^{n}$
and constructed
cohomogeneity one special Lagrangian submanifolds in $T^{*}S^{n}\cong Q^{n}$.
In construction,
a moment map $\mu: Q^{n}\to \mathfrak{k}^{*}$
for the natural $K$-action is given by
$\mu(z)X=\alpha_{Stz}'(X^{*}_{z})\,(X\in\mathfrak{k}, z\in Q^{n})$,
where $\alpha_{Stz}'$ denotes the Liouville 1-form for $\varrho_{Stz}'$.
\end{Rem}

\subsection{A cohomogeneity one action on $\mathbb{C}P^{n}$}

The classification of cohomogeneity one actions on the complex projective space was given by Takagi (\cite{T}).
According to his classification,
they are obtained essentially as the linear isotropy actions
of rank two Hermitian symmetric spaces.
In Table \ref{table:HSS},
we list rank two Hermitian orthogonal symmetric Lie algebras
of compact type.

\begin{table}[!!h]
\centering
\caption{The classification of rank two Hermitian orthogonal symmetric Lie algebras $(\mathfrak{u},\mathfrak{k})$ of compact type, where $(\mathfrak{u},\mathfrak{k})$ is an effective rank two Hermitian orthogonal symmetric Lie algebra of compact type.}\label{table:HSS}
\begin{tabular}{|c|c|c|c|}
\hline
Type&$\mathfrak{u}$ & $\mathfrak{k}$ & Remark \\\hline
AIII$+$AIII &$\mathfrak{su}(p+1)\oplus\mathfrak{su}(q+1)$ & $\mathfrak{s}(\mathfrak{u}(p)\oplus\mathfrak{u}(1))\oplus\mathfrak{s}(\mathfrak{u}(q)\oplus\mathfrak{u}(1))$ & $p\geq q\geq 1, p>1$\\\hline
AIII & $\mathfrak{su}(m+2)$ & $\mathfrak{s}(\mathfrak{u}(m)\oplus\mathfrak{u}(2))$ & $m\geq 3$ \\\hline
BDI &$\mathfrak{o}(m+2)$ & $\mathfrak{o}(m)\oplus\mathfrak{o}(2)$ &$m\geq 3$ \\\hline
DIII & $\mathfrak{o}(10)$ & $\mathfrak{u}(5)$ & \\\hline
EIII & $\mathfrak{e}_{6}$ & $\mathfrak{o}(10)\oplus\mathfrak{o}(2)$ & \\\hline
\end{tabular}
\end{table}

Notations of Lie algebras in Table \ref{table:HSS} are as follows:
\begin{itemize}
\item $\mathfrak{u}(n)$ : the set of all skew Hermitian matrices of order $n$,
\item $\mathfrak{su}(n)$ : the set of all traceless skew Hermitian matrices of order $n$,
\item $\mathfrak{s}(\mathfrak{u}(p)\oplus\mathfrak{u}(q)):=\left\{\begin{pmatrix}
A & 0\\
0 & B
\end{pmatrix}\in\mathfrak{su}(p+q)\biggm|
A\in\mathfrak{u}(p),B\in\mathfrak{u}(q),
\mathrm{Tr}\,A+\mathrm{Tr}\,B=0
\right\}$,
\item $\mathfrak{o}(n)$ : the set of skew symmetric matrices of order $n$,
\item $\mathfrak{e}_{6}$ : the compact simple Lie algebra with root type
$E_{6}$.
\end{itemize}

\medskip

In the following,
we explain the linear isotropy action
of rank two Hermitian symmetric spaces
and a cohomogeneity one action on
complex projective spaces induced from this action.

Let $(\mathfrak{u},\mathfrak{k})$ be an effective rank two Hermitian orthogonal symmetric Lie algebra of compact type, and
$\theta$ be an involutive automorphism of $\mathfrak{u}$ with $\mathfrak{k}=\mathrm{Ker}(\theta-\mathrm{id})$.
Setting $\mathfrak{p}=\mathrm{Ker}(\theta+\mathrm{id})$,
we have $\mathfrak{u}=\mathfrak{k}\oplus\mathfrak{p}$,
$[\mathfrak{k},\mathfrak{k}]\subset\mathfrak{k}, [\mathfrak{k},\mathfrak{p}]\subset\mathfrak{p}$ and $[\mathfrak{p},\mathfrak{p}]\subset\mathfrak{k}$.
We define a positive definite inner product
on $\mathfrak{p}$, which is induced from the
Killing form of $\mathfrak{u}$.
Let $K$ be an analytic subgroup of
the group $U$ of inner automorphisms of $\mathfrak{u}$
and the corresponding Lie algebra is $\mathrm{ad}(\mathfrak{k})$,
where $\mathrm{ad}$ represents the adjoint representation of $\mathfrak{u}$.
Defining an orthogonal representation $\rho: K \to O(\mathfrak{p})$ by
$\rho(k)x=k(x)$ for $k\in K$ and $x\in \mathfrak{p}$,
we have a $K$-action on $\mathfrak{p}$.
This is called the \textit{linear isotropy action}
of $U/K$ (or $(\mathfrak{u},\mathfrak{k})$).
It is shown that this $K$-action is polar
and a maximal abelian subspace
$\mathfrak{a}$ in $\mathfrak{p}$ gives a section of the $K$-action,
that is,
$\mathfrak{a}$ meets all $K$-orbits and is perpendicular to the $K$-orbits at the points of intersection (cf.~\cite{BCO}).
This implies that the orbit space for the $K$-orbit on $\mathfrak{p}$
is parametrized by elements of the quotient set $\mathfrak{a}/W$,
where $W(\subset GL(\mathfrak{a}))$ denotes the Weyl group of the restricted root system
of $(\mathfrak{u},\mathfrak{k})$ with respect to $\mathfrak{a}$.
It should be emphasized
that the cohomogeneity of the $K$-action on $\mathfrak{p}$
is equal to two,
which is equal to the rank of $(\mathfrak{u},\mathfrak{k})$.

We give cohomogeneity one action on complex projective spaces as follows.
Let $J_{0}$ be an element in the center of $\mathfrak{k}$
such that $\rho_{*}(J_{0})$ is a complex structure on $\mathfrak{p}$.
This allows us to identify $\mathfrak{p}$ with
an $(n+1)$-dimensional complex vector space
$\mathbb{C}^{n+1}$.
The definition of $J_{0}$ gives $\rho(K)\subset SU(n+1)$.
Let $S^{2n+1}$ be the unit hypersphere in $\mathfrak{p}$
centered at the origin.
Then the $K$-action on $\mathfrak{p}$
naturally induces a $K$-action on $S^{2n+1}$,
which we write the same symbol $\rho$ as before: $\rho(k)x=k(x)$
for $k\in K$ and $x\in S^{2n+1}$.
From the above, it is seen that
the $K$-action on $S^{2n+1}$
is isometric and cohomogeneity one.
This implies that the $K$-action is polar
and $\mathfrak{a}\cap S^{2n+1}$ gives a section.
We note that,
for each $X\in\mathfrak{k}$,
the fundamental vector field
$X^{*}$ is expressed by $X^{*}_{x}=\rho_{*}(X)x=\mathrm{ad}(X)x$
for $x\in S^{2n+1}$.
The $K$-action on $S^{2n+1}$ yields
the complex projective space $\mathbb{C}P^{n}$ 
with standard Fubini-Study metric of constant holomorphic sectional curvature four.
Let $\pi$ be the canonical projection from $\mathfrak{p}-\{0\}=\mathbb{C}^{n+1}-\{0\}$ onto $\mathbb{C}P^{n}$.
We use the same symbol to represent the restriction of $\pi$ to $S^{2n+1}$,
which is called \textit{Hopf fibration}.
This gives a Riemannian submersion.
Here we note that
the $1$-parameter transformation group $\{\rho(\exp tJ_{0})\}_{t\in\mathbb{R}}
(\subset \rho(K))$
determines an $S^{1}$-fiber of $\pi:S^{2n+1}\to\mathbb{C}P^{n}$.
Defining a Lie homomorphism
from  $K$ to the isometry group of $\mathbb{C}P^{n}$
which we write the same symbol $\rho$,
by $\rho(k)\pi(x)=\pi(\rho(k)x)$ for $k\in K$ and $x\in S^{2n+1}$,
the tangent space
of the $K$-orbit at $\pi(x)$
coincides with $\pi_{*}(\mathrm{ad}(\mathfrak{k})x)$.
From the above,
we can verify the $K$-action on $\mathbb{C}P^{n}$
is cohomogeneity one.
This shows that the $K$-action is polar and $\pi(\mathfrak{a}\cap S^{2n+1})$
gives a section.

\begin{Rem}
The classification of cohomogeneity one actions on the sphere
was given by Hsiang and Lawson (\cite{HsLa}).
It follows from their classification that
any cohomogeneity one action on the sphere
is realized as the linear isotropy action of rank two
Riemannian symmetric spaces.
\end{Rem}

\subsection{A Hamilton action on $T^{*}\mathbb{C}P^{n}$}\label{subsec:ConstHamaction}
Let $p'$ (resp.~$p$)
be the natural projection of $TS^{2n+1}$ (resp.~$T\mathbb{C}P^{n}$)
and $d\pi:TS^{2n+1}\to T\mathbb{C}P^{n}$ denote the differential of $\pi$.
We define a map $\Pi:Q^{2n+1}\to \mathcal{M}$ by $\Pi=\Phi\circ d\pi \circ \Phi'^{-1}$.
The above setting is summarized as
the following commutative diagram:
\begin{equation*}
\begin{CD}
S^{2n+1} @<p'<< TS^{2n+1} @>\Phi'>> Q^{2n+1}\\
@VV\pi V @VVd\pi V @VV \Pi V \\
\mathbb{C}P^{n} @<p<< T\mathbb{C}P^{n}@>\Phi>>\mathcal{M}
\end{CD}
\end{equation*}

\noindent
The $K$-action on $S^{2n+1}$
(resp.~$\mathbb{C}P^{n}$)
is naturally extended to
a $K$-action on $TS^{2n+1}$
(resp.~$T\mathbb{C}P^{n}$),
which we write $\tilde{\rho}'$ (resp.~$\tilde{\rho}$).
We define $K$-actions $\hat{\rho}'$ and $\hat{\rho}$
on $Q^{2n+1}$ and $\mathcal{M}$ as follows,
respectively:
\begin{equation*}
\hat{\rho}'(k)z=\rho(k)x+i\rho(k)y,\quad
\hat{\rho}(k)(w_{1};w_{2})=(\rho(k)w_{1};\bar{\rho}(k)w_{2})
\end{equation*}
for $k\in K$,
$z=x+iy\in Q^{2n+1}(\subset \mathbb{C}^{2n+2})$ and
$(w_{1};w_{2})\in\mathcal{M}(\subset \mathbb{C}P^{n}\times\mathbb{C}P^{n})$,
where $\bar{\rho}$ is the complex conjugate of $\rho$.
A simple calculation leads to the following result.
\begin{Lem}\label{lem:Kequiv}
The maps
$d\pi$,
$\Phi'$, $\Phi$
and $\Pi$
are $K$-equivariant maps,
respectively$:$
\begin{equation*}\label{eqn:actionequiv}
\tilde{\rho}(k)\circ d\pi=d\pi \circ \tilde{\rho}'(k),\quad
\hat{\rho}'(k)\circ\Phi'=\Phi'\circ\tilde{\rho}'(k),\quad
\hat{\rho}(k)\circ\Phi=\Phi\circ\tilde{\rho}(k),\quad
\hat{\rho}(k)\circ\Pi=\Pi\circ\hat{\rho}'(k)\quad
(k\in K).
\end{equation*}
\end{Lem}
\noindent
Moreover,
the $K$-action on $T\mathbb{C}P^{n}\cong\mathcal{M}$
preserves the Calabi-Yau structure $(J_{Stz},\omega_{Stz},\Omega_{Stz})$
invariantly.
We define a $K$-equivariant mapping
$\hat{\mu}: \mathcal{M}\to \mathfrak{k}^{*}$
by $\INN{\hat{\mu}(z;w)}{X}=\alpha_{Stz}(X^{*}_{(z;w)})$
for $(z;w)\in \mathcal{M}$ and $X\in \mathfrak{k}$.
\begin{Pro}\label{pro:HamKM}
The $K$-action $\hat{\rho}$ on $T\mathbb{C}P^{n}\cong\mathcal{M}$ is
a Hamiltonian action with moment map $\hat{\mu}$.
\end{Pro}
\begin{proof}
For each $X\in \mathfrak{k}$,
we define $\hat{\mu}_{X}\in C^{\infty}(\mathcal{M})$ by $\hat{\mu}_{X}(z;w)=\INN{\hat{\mu}(z;w)}{X}$
for $(z;w)\in\mathcal{M}$.
By definition of $\hat{\mu}$
we have $\hat{\mu}_{X}(z;w)=\alpha_{Stz}(X^{*}_{(z;w)})$
for $X\in \mathfrak{k}$ and $(z;w)\in\mathcal{M}$.
From Cartan's magic formula
we obtain
\begin{equation*}
d\hat{\mu}_{X}
=d\iota(X^{*})\alpha_{Stz}
=-\iota(X^{*})d\alpha_{Stz}+\mathcal{L}_{X^{*}}\alpha_{Stz}
=\iota(X^{*})\omega_{Stz}+\mathcal{L}_{X^{*}}\alpha_{Stz}\quad
(X\in \mathfrak{k}).
\end{equation*}
In addition, for any $X\in \mathfrak{k}$, $\mathcal{L}_{X^{*}}\alpha_{Stz}=0$ holds
because of $\hat{\rho}(K)\subset SU(n+1)$.
This implies that $\hat{\mu}:\mathcal{M}\to\mathfrak{k}^{*}$
is a moment map of $\hat{\rho}$.
\end{proof}
\noindent
In Section \ref{sec:constSLag},
we will give cohomogeneity one special Lagrangian submanifolds
in $\mathcal{M}$ by applying Proposition \ref{pro:HSpro2} to the Hamiltonian action $\hat{\rho}$.
As we shall see in Section \ref{subsec:AIIIAIII}--\ref{subsec:DIII},
the cohomogeneity one special Lagrangian submanifolds are contained
in $\hat{\mu}^{-1}(c=0)=\{(z;w)\in\mathcal{M}\mid \alpha_{Stz}(X^{*}_{(z;w)})=0,\,\forall X\in\mathfrak{k}\}$.
Indeed,
we solve
the equations $\alpha_{Stz}(X^{*}_{(z;w)})=0\,(\forall X\in\mathfrak{k})$
by means of
the formula (\ref{eqn:alphastz})
and a suitable basis of $\mathfrak{k}$.

\begin{Rem}\label{rem:slagTS2}
It is well-known that
$L=S^{2n+1}$ (regarded as $0$-section of $T^{*}S^{n}$) is a special Lagrangian
submanifold in $T^{*}S^{2n+1}$
and $\mathbb{C}P^{n}=d\pi(L)$ is also
a special Lagrangian submanifold in $T^{*}\mathbb{C}P^{n}$.
This is a particular case when
$d\pi$ preserves the special Lagrangianness of $L$.
However,
in general,
the image of a special Lagrangian submanifold in $T^{*}S^{2n+1}$ under $d\pi$ may not be special Lagrangian.
In fact,
for the cohomogeneity one
special Lagrangian submanifold $L(\subset Q^{3}\cong T^{*}S^{3})$
which is given in \cite[Theorem 3.6]{HS},
we can show that $\Pi(L)$ is not special Lagrangian in $\mathcal{M}\cong T^{*}\mathbb{C}P^{1}$, which is also identified with $Q^{2}\cong T^{*}S^{2}$ via the stereographic projection.
\end{Rem}

\section{Construction of cohomogeneity one special Lagrangian submanifolds in $T^{*}\mathbb{C}P^{n}$}\label{sec:constSLag}

In this section,
we will construct special Lagrangian submanifolds
in $T^{*}\mathbb{C}P^{n}\cong \mathcal{M}$
by means of cohomogeneity one actions
on $\mathbb{C}P^{n}$
induced from the linear isotropy actions of rank two classical Hermitian symmetric spaces.

\medskip

We summarize basic notations that are used throughout this section:
\begin{itemize}
\item $M_{n,m}(\mathbb{C})$ (resp.~$M_{n,m}(\mathbb{R})$) : the set of
complex (resp.~real) $n\times m$ matrices,
\item $E_{n}$ : the unit matrix of order $n$,
\item $E_{ij}^{(n)}$ : the matrix $(\delta_{ai}\delta_{bj})_{1\leq a,b\leq n}$.
(Here $\delta_{ab}=1$ if $a=b$, 0 otherwise.)
\item The transpose and the complex conjugate of a matrix $X$ are denoted by
${}^{t}X$ and $\bar{X}$, respectively.
\end{itemize}

\subsection{Case of $(\mathfrak{u},\mathfrak{k})=(\mathfrak{su}(p+1)\oplus\mathfrak{su}(q+1),\mathfrak{s}(\mathfrak{u}(p)\oplus\mathfrak{u}(1))\oplus\mathfrak{s}(\mathfrak{u}(q)\oplus\mathfrak{u}(1)))$}\label{subsec:AIIIAIII}

In this subsection,
we adopt the following notations:
\begin{itemize}
\item $\mathfrak{u}=\left\{\left(\begin{array}{c|c}
X & 0 \\\hline
0 & X'
\end{array}\right)\biggm| X\in\mathfrak{su}(p+1), X'\in\mathfrak{su}(q+1)\right\}$
\item $\mathfrak{k}=\left\{\left(\begin{array}{c|c}
\begin{array}{c|c}
X_{1} &       \\\hline
      & X_{2}
\end{array} &  \\\hline
 & \begin{array}{c|c}
X_{1}' &       \\\hline
      & X_{2}'
\end{array}
\end{array}\right)\in\mathfrak{u}\biggm| 
\begin{array}{c}
X_{1}\in\mathfrak{u}(1),X_{2}\in\mathfrak{u}(p),\mathrm{Tr}\,X_{1}+\mathrm{Tr}\,X_{2}=0\\
X_{1}'\in\mathfrak{u}(1),X_{2}'\in\mathfrak{u}(q),\mathrm{Tr}\,X_{1}'+\mathrm{Tr}\,X_{2}'=0
\end{array}
\right\}$
\item $\mathfrak{p}=\left\{\left(\begin{array}{c|c}
\begin{array}{c|c}
 & Z      \\\hline
-{}^{t}\bar{Z}\rule{0pt}{2.5ex}      & 
\end{array} &  \\\hline
 & \begin{array}{c|c}
 & Z'      \\\hline
-{}^{t}\bar{Z'}\rule{0pt}{2.5ex}      & 
\end{array}
\end{array}\right)\in\mathfrak{u}\biggm| 
\begin{array}{c}
Z: \text{complex $p$-row vector}\\
Z': \text{complex $q$-row vector}
\end{array}
\right\}$
\item $\mathfrak{a}$ is the set of all elements which have the form
$\left(\begin{array}{c|c}
\begin{array}{c|c}
 & A      \\\hline
-{}^{t}A      & 
\end{array} &  \\\hline
 & \begin{array}{c|c}
 & A'      \\\hline
-{}^{t}A'      & 
\end{array}
\end{array}\right)\in\mathfrak{p}$ with
\begin{equation*}
A=(a_{1},0,\ldots,0)\in M_{1,p}(\mathbb{R}),\quad
A'=(a_{1}',0,\ldots,0)\in M_{1,q}(\mathbb{R}),
\end{equation*}
\end{itemize}
where every blank entry is zero.
In this setting, it is seen that $\mathfrak{a}$ is a maximal abelian subspace of $\mathfrak{p}$.
The following element $J_{0}$ in the center of $\mathfrak{k}$
gives the complex structure $\mathrm{ad}(J_{0})$ on $\mathfrak{p}$:
\begin{equation*}
J_{0}= \left(\begin{array}{c|c}
J_{0,p} & 0\\\hline
0 & J_{0,q}'
\end{array}\right),\quad
J_{0,p}=\dfrac{i}{p+1}\left(\begin{array}{c|c}
p & \\\hline
  & -E_{p}
\end{array}\right),\quad
J_{0,q}'=\dfrac{i}{q+1}\left(\begin{array}{c|c}
q & \\\hline
  & -E_{q}
\end{array}\right).
\end{equation*}
Therefore we naturally identify $(\mathfrak{p},\mathrm{ad}(J_{0}))$ with
$\mathbb{C}^{p+q}$.
The sphere $S^{2(p+q)-1}$ can be regarded as a subset of $\mathbb{C}^{p+q}$:
\begin{equation*}
S^{2(p+q)-1}= \left\{
\left(\begin{array}{c|c}
\begin{array}{c|c}
 & Z      \\\hline
-{}^{t}\bar{Z}\rule{0pt}{2.5ex}      & 
\end{array} &  \\\hline
 & \begin{array}{c|c}
 & Z'      \\\hline
-{}^{t}\bar{Z'}\rule{0pt}{2.5ex}      & 
\end{array}
\end{array}\right)\biggm|
\begin{array}{c}
Z=(z_{1},\ldots,z_{p})\in M_{1,p}(\mathbb{C})\\
Z'=(z_{1}',\ldots,z_{q}')\in M_{1,q}(\mathbb{C})\\
|z_{1}|^{2}+\dotsb+|z_{p}|^{2}+|z_{1}'|^{1}+\dotsb+|z_{q}'|^{2}=1
\end{array}
\right\}.
\end{equation*}
Using the notation above,
$\mathfrak{a}\cap S^{2(p+q)-1}$ consists of matrices of the form
\begin{equation*}
\left(\begin{array}{c|c}
\begin{array}{c|c}
 & A_{\theta}      \\\hline
-{}^{t}A_{\theta}      & 
\end{array} &  \\\hline
 & \begin{array}{c|c}
 & A'_{\theta}      \\\hline
-{}^{t}A'_{\theta}      & 
\end{array}
\end{array}\right)
\end{equation*}
with
\begin{equation*}
A_{\theta}=(\cos\theta,0,\ldots,0)\in M_{1,p}(\mathbb{R}),\quad
A'_{\theta}=(\sin\theta,0,\ldots,0)\in M_{1,q}(\mathbb{R}),\quad\theta\in\mathbb{R}.
\end{equation*}
We consider elements
$Z_{ij}, W_{ij}\,(1\leq i<j \leq p),W_{k}\,(1\leq k\leq p-1), J_{1}$
in $\mathfrak{k}$ defined as follows:
\begin{align*}
Z_{ij}&=\left(\begin{array}{c|c}
\begin{array}{c|c}
0 & \\\hline
&-E_{ij}^{(p)}+E_{ji}^{(p)}
\end{array} & \\\hline
 &0
\end{array}\right),\\
W_{ij}&=\left(\begin{array}{c|c}
\begin{array}{c|c}
0 & \\\hline
&i(E_{ij}^{(p)}+E_{ji}^{(p)})
\end{array} & \\\hline
 &0
\end{array}\right),\quad
W_{k}=\left(\begin{array}{c|c}
\begin{array}{c|c}
0 & \\\hline
&i(E_{kk}^{(p)}-E_{k+1,k+1}^{(p)})
\end{array} & \\\hline
 &0
\end{array}\right),\\
J_{1}&=\left(\begin{array}{c|c}
J_{0,p} &\\\hline
&-J_{0,q}'
\end{array}\right).
\end{align*}
If $q>1$,
we also consider elements
$Z_{st}', W_{st}'\,(1\leq s<t \leq q),W_{l}'\,(1\leq l\leq q-1)$ in
$\mathfrak{k}$ as follows:
\begin{align*}
Z'_{st}&=\left(\begin{array}{c|c}
0 & \\\hline
 &\begin{array}{c|c}
0 & \\\hline
&-E_{st}^{(q)}+E_{ts}^{(q)}
\end{array}
\end{array}\right),\\
W_{st}'&=\left(\begin{array}{c|c}
0 & \\\hline
& \begin{array}{c|c}
0 & \\\hline
& i(E_{st}^{(q)}+E_{ts}^{(q)})
\end{array}
\end{array}\right),\quad
W_{l}'=\left(\begin{array}{c|c}
0 & \\\hline
& \begin{array}{c|c}
0 & \\\hline
& i(E_{ll}^{(q)}-E_{l+1,l+1}^{(q)})
\end{array}
\end{array}\right).
\end{align*}
Then $Z_{ij}, W_{ij}\,(1\leq i<j \leq p),W_{k}\,(1\leq k\leq p-1), Z_{st}', W_{st}'\,(1\leq s<t \leq q),W_{l}'\,(1\leq l\leq q-1)$ and $J_{a}\,(a=0,1)$
give a basis of $\mathfrak{k}$.
The $K$-orbit can be directly calculated
through
$\tilde{A}=\pi(A)\in \mathbb{C}P^{p+q-1}$
with
\begin{equation*}
A=\left(\begin{array}{c|c}
\begin{array}{c|c}
 & A_{\theta}      \\\hline
-{}^{t}A_{\theta}      & 
\end{array} &  \\\hline
 & \begin{array}{c|c}
 & A'_{\theta}      \\\hline
-{}^{t}A'_{\theta}      & 
\end{array}
\end{array}\right)\textcolor{red}{,}
\end{equation*}
where $A$ has codimension one if and only if
$\theta\not\in(\pi/2)\mathbb{Z}$.
We see that the tangent space at $\tilde{A}$
is spanned by
\begin{equation*}
(Z_{1j})^{*}_{\tilde{A}},\,(W_{1j})^{*}_{\tilde{A}}\,(2\leq j\leq p),\quad
(J_{1}')^{*}_{\tilde{A}},\quad
(Z_{1t}')^{*}_{\tilde{A}},\,(W_{1t}')^{*}_{\tilde{A}}\,(2\leq t\leq q),
\end{equation*}
where $\pi:\mathbb{C}^{p+q}-\{0\}\to\mathbb{C}P^{p+q-1}$ denotes
the canonical projection.
In this subsection,
we prove the following result.

\begin{Thm}\label{thm:slagAIIIAIII}
Let $\tau:I\to\{z\in\mathbb{C}\mid 0 < |\mathrm{Re}\,z|<\pi/2\}(\subset\mathbb{C})$ be a regular curve and $\sigma: I\to \mathcal{M}(\cong \Phi(T\mathbb{C}P^{p+q-1})$ be the curve defined by
\begin{equation*}
\sigma=(\left(\begin{array}{c|c}
\begin{array}{c|c}
 & A      \\\hline
-{}^{t}\bar{A}\rule{0pt}{2.5ex}      & 
\end{array} &  \\\hline
 & \begin{array}{c|c}
 & A'      \\\hline
-{}^{t}\bar{A'}\rule{0pt}{2.5ex}      & 
\end{array}
\end{array}\right);\left(\begin{array}{c|c}
\begin{array}{c|c}
 & A      \\\hline
-{}^{t}\bar{A}\rule{0pt}{2.5ex}      & 
\end{array} &  \\\hline
 & \begin{array}{c|c}
 & A'      \\\hline
-{}^{t}\bar{A'}\rule{0pt}{2.5ex}      & 
\end{array}
\end{array}\right)),
\end{equation*}
where $A=(\cos\tau(s),0,\dotsc,0)$
and $A'=(\sin\tau(s),0,\dotsc,0)$.
Then $L:=\{\hat{\rho}(k)\sigma(s)\mid k\in K,s\in I\}$
is a cohomogeneity one Lagrangian submanifold in $\mathcal{M}$.
Moreover,
if $\tau$ is a solution of the following ODE $(\ref{eqn:AIII+AIIIspecial})$,
then $L$ is a special Lagrangian submanifold with phase $\psi$$:$
\begin{equation}\label{eqn:AIII+AIIIspecial}
\mathrm{Im}(e^{i\psi}i^{p+q-1}\tau'(1+\tan^{2}\tau)(\tan\tau)^{2q-1})=0.
\end{equation}
\end{Thm}
\begin{proof}
Set $B=\{(\zeta,\xi)\in\mathbb{C}^{p+q}\times\mathbb{C}^{p+q}\mid \zeta\neq0,\xi\cdot \bar{\zeta}=0\}$
and
\begin{equation*}
\zeta=\left(\begin{array}{c|c}
\begin{array}{c|c}
 & A_{\theta}      \\\hline
-{}^{t}A_{\theta}      & 
\end{array} &  \\\hline
 & \begin{array}{c|c}
 & A'_{\theta}      \\\hline
-{}^{t}A'_{\theta}      & 
\end{array}
\end{array}\right),\quad \theta\not\in\dfrac{\pi}{2}\mathbb{Z}.
\end{equation*}
If $(\zeta,\xi)\in B$ holds,
then $\xi\in\mathbb{C}^{p+q}$ has the form:
\begin{equation*}
\xi=\left(\begin{array}{c|c}
\begin{array}{c|c}
 & X      \\\hline
-{}^{t}\bar{X}\rule{0pt}{2.5ex}      & 
\end{array} &  \\\hline
 & \begin{array}{c|c}
 & X'      \\\hline
-{}^{t}\bar{X'}\rule{0pt}{2.5ex}      & 
\end{array}
\end{array}\right),
X=(\alpha\sin\theta,\xi_{2},\ldots,\xi_{p}),
X'=(-\alpha\cos\theta,\xi_{2}',\ldots,\xi_{q}'),
\alpha,\xi_{i},\xi_{s}'\in\mathbb{C}.
\end{equation*}
Let $(z;w)=\Phi(\zeta,\xi)$.
With the use of (\ref{eqn:alphastz}), we can verify
that
\begin{align*}
\alpha&_{Stz}(X^{*}_{(z;w)})=0\quad(\forall X\in\mathfrak{k})\\
&\Longleftrightarrow
\alpha_{Stz}(Z_{1j})^{*}_{(z;w)}=0,
\alpha_{Stz}(W_{1j})^{*}_{(z;w)}=0,
\alpha_{Stz}(J_{1})^{*}_{(z;w)}=0,
\alpha_{Stz}(Z'_{1t})^{*}_{(z;w)}=0,
\alpha_{Stz}(W'_{1t})^{*}_{(z;w)}=0\\
&\Longleftrightarrow
\xi_{2}=\dotsb=\xi_{p}=0,
\xi_{2}'=\dotsb=\xi_{q}'=0,
\alpha=\lambda\in\mathbb{R},
\end{align*}
where we have used the condition
$\theta\not\in(\pi/2)\mathbb{Z}$.
Let $\tau=\tau(s)$ be a regular curve
in $\{z\in\mathbb{C}\mid 0<|\mathrm{Re}\,z|<\pi/2\}(\subset \mathbb{C})$
and
$\sigma=\sigma(s)$ be
the curve in $\mathcal{M}$
defined by
\begin{equation*}
\sigma=(\left(\begin{array}{c|c}
\begin{array}{c|c}
 & A      \\\hline
-{}^{t}\bar{A}\rule{0pt}{2.5ex}      & 
\end{array} &  \\\hline
 & \begin{array}{c|c}
 & A'      \\\hline
-{}^{t}\bar{A'}\rule{0pt}{2.5ex}      & 
\end{array}
\end{array}\right);\left(\begin{array}{c|c}
\begin{array}{c|c}
 & A      \\\hline
-{}^{t}\bar{A}\rule{0pt}{2.5ex}      & 
\end{array} &  \\\hline
 & \begin{array}{c|c}
 & A'      \\\hline
-{}^{t}\bar{A'}\rule{0pt}{2.5ex}      & 
\end{array}
\end{array}\right)),
\end{equation*}
where $A=(\cos\tau(s),0,\dotsc,0)$
and $A'=(\sin\tau(s),0,\dotsc,0)$.
Then, the curve $\sigma$ is contained in $\hat{\mu}^{-1}(0)$,
Hence it follows from Proposition \ref{pro:HSpro2} that 
$L:=\{\hat{\rho}(k)\sigma(s)\mid k\in K,s\in I\}$ is a cohomogeneity one Lagrangian
submanifold in $\mathcal{M}$.
Direct calculation yields
\begin{equation*}
\begin{split}
\Omega_{Stz}&(\sigma',(Z_{12})^{*}_{\sigma},\dotsc,(Z_{1p})^{*}_{\sigma},
(W_{12})^{*}_{\sigma},\dotsc,(W_{1q})^{*}_{\sigma},
(J_{1})^{*}_{\sigma},
(Z_{12}')^{*}_{\sigma},\dotsc,(Z_{1q}')^{*}_{\sigma},
(W_{12}')^{*}_{\sigma},\dotsc,(W_{1q}')^{*}_{\sigma})\\
&=
(-1)^{pq-1}2^{p+q}i^{p+q-1}\tau'(1+\tan^{2}\tau)(\tan\tau)^{2q-1}.
\end{split}
\end{equation*}
From Proposition \ref{pro:HLpro}
it is shown that
$L$ is a special Lagrangian submanifold
(with phase $\psi$)
if $\tau$ is a solution of $(\ref{eqn:AIII+AIIIspecial})$.
\end{proof}

\subsection{Case of $(\mathfrak{u},\mathfrak{k})=(\mathfrak{su}(m+2),\mathfrak{s}(\mathfrak{u}(m)\oplus\mathfrak{u}(2)))$}\label{subsec:HSSAIII}

In this subsection,
we adopt the following notations (\cite[p.~452]{Hel}):
\begin{itemize}
\item $\mathfrak{u}=\mathfrak{su}(m+2)$
\item $\mathfrak{k}=\left\{\begin{pmatrix}
X_{1} & 0 \\
0 & X_{2}
\end{pmatrix}\in\mathfrak{u}\biggm| X_{1}\in\mathfrak{u}(2), X_{2}\in\mathfrak{u}(m), \mathrm{Tr}\,X_{1}+\mathrm{Tr}\,X_{2}=0\right\}$
\item $\mathfrak{p}=\left\{\begin{pmatrix}
0 & X \\
-{}^{t}\bar{X} & 0
\end{pmatrix}\in\mathfrak{u}
\biggm| X : \text{$(2\times m)$-complex matrix}\right\}$
\item $\mathfrak{a}=\left\{\begin{pmatrix}
0 & A \\
-{}^{t}A & 0
\end{pmatrix}\in\mathfrak{p}\biggm|
A=\begin{pmatrix}
a_{1} & 0 & 0 & \dotsb & 0\\
0 & a_{2} & 0 & \dotsb & 0
\end{pmatrix}, a_{1},a_{2}\in\mathbb{R}\right\}
$
\end{itemize}
In this setting, $\mathfrak{a}$
is a maximal abelian subspace of $\mathfrak{p}$.
The following element $J_{0}$ in the center of $\mathfrak{k}$
gives the complex structure $\mathrm{ad}(J_{0})$ on $\mathfrak{p}$:
\begin{equation*}
J_{0}=\dfrac{i}{m+2}\left(\begin{array}{cc}
m E_{2} & 0 \\
0 & -2E_{m}
\end{array}\right).
\end{equation*}
Therefore we naturally identify $(\mathfrak{p},\mathrm{ad}(J_{0}))$ with $\mathbb{C}^{2m}$.
The sphere $S^{4m-1}$ can be regarded as a subset of
$\mathbb{C}^{2m}$:
\begin{equation*}
S^{4m-1}=\left\{
\left(\begin{array}{cc}
0 & Z \\
-{}^{t}\bar{Z} & 0
\end{array}\right)
\biggm|
\begin{array}{c}
Z=\left(\begin{array}{ccc}
z_{11} & \dotsb & z_{1m}\\
z_{21} & \dotsb & z_{2m}\\
\end{array}\right):(2\times m)\text{-complex matrix}\\
|z_{11}|^{2}+\dotsb+|z_{1m}|^{2}+|z_{21}|^{2}+\dotsb+|z_{2m}|^{2}=1
\end{array}
\right\}.
\end{equation*}
Using the notation above,
it is shown that $\mathfrak{a}\cap S^{4m-1}$
consists of matrices of the form
$
\begin{pmatrix}
0 & A_{\theta} \\
-{}^{t}A_{\theta} & 0
\end{pmatrix}
$ with
\begin{equation*}
A_{\theta}=\left(\begin{array}{ccccc}
\cos\theta & 0 & 0 & \dotsb & 0\\
0 & \sin\theta & 0 & \dotsb & 0 \\
\end{array}\right),\quad \theta\in\mathbb{R}.
\end{equation*}
We consider elements $X_{l}\,(1\leq l\leq 3)$,
$Z_{ij}, W_{ij}\,(1\leq i < j \leq m)$ and $W_{k}\,(1\leq k\leq m-1)$
in $\mathfrak{k}$
defined as follows:
\begin{align*}
X_{1}&=\left(\begin{array}{cc|ccc}
0  & -1 \\
1 & 0 \\\hline
&&&&\\
&&&0&\\
&&&&\\
\end{array}\right),\quad
X_{2}=\left(\begin{array}{cc|ccc}
0  & i \\
i & 0 \\\hline
&&&&\\
&&&0&\\
&&&&\\
\end{array}\right),\quad
X_{3}=\left(\begin{array}{cc|ccc}
i &0 \\
0 & -i \\\hline
&&&&\\
&&&0&\\
&&&&\\
\end{array}\right),\\
Z_{ij}&=\left(\begin{array}{c|ccc}
0\\\hline
\\
&&-E_{ij}^{(m)}+E_{ji}^{(m)}&\\
\\
\end{array}\right),\\
W_{ij}&=\left(\begin{array}{c|ccc}
0\\\hline
\\
&&i(E_{ij}^{(m)}+E_{ji}^{(m)})&\\
\\
\end{array}\right),\quad
W_{k}=\left(\begin{array}{c|ccc}
0\\\hline
\\
&&i(E_{kk}^{(m)}-E_{k+1,k+1}^{(m)})&\\
\\
\end{array}\right).
\end{align*}
Then
$X_{l}\,(1\leq l\leq 3)$,
$Z_{ij}, W_{ij}\,(1\leq i < j \leq m)$, $W_{k}\,(1\leq k\leq m-1)$
and $J_{0}$ give a basis of $\mathfrak{k}$.
We can confirm that
the $K$-orbit through $\tilde{A}=\pi(A)\in\mathbb{C}P^{2m-1}$ with
$A=\begin{pmatrix}
0 & A_{\theta} \\
-{}^{t}A_{\theta} & 0
\end{pmatrix}$
has codimension one if and only if $\theta\not\in(\pi/4)\mathbb{Z}$.
The tangent space at $\tilde{A}$ is spanned by
\begin{equation*}
(X_{l})^{*}_{\tilde{A}}\,(l=1,2),\quad
(Z_{ij})^{*}_{\tilde{A}}, (W_{ij})^{*}_{\tilde{A}}\,(i=1,2, i < j \leq m),\quad
(W_{2})^{*}_{\tilde{A}},
\end{equation*}
where $\pi:\mathbb{C}^{2m}-\{0\}\to \mathbb{C}P^{2m-1}$
denotes the canonical projection.
In this subsection
we prove the following result.
\begin{Thm}\label{thm:slagAIII}
Let $\tau:I \to \{z\in\mathbb{C} \mid 0 < |\mathrm{Re}\,z|<\pi/4\}(\subset \mathbb{C})$ be a regular curve
and $\sigma:I \to \mathcal{M}(\cong\Phi(T\mathbb{C}P^{2m-1}))$
be the curve defined by
\begin{equation*}
\sigma=(\left(\begin{array}{cc}
0 & X\\
-{}^{t}\bar{X}&0
\end{array}\right);\left(\begin{array}{cc}
0 & X\\
-{}^{t}\bar{X}&0
\end{array}\right))\quad
X=\left(\begin{array}{ccccc}
\cos\tau(s) & 0 & 0 & \dotsb & 0\\
0 & \sin\tau(s) & 0 & \dotsb & 0
\end{array}\right).
\end{equation*}
Then $L:=\{\hat{\rho}(k)\sigma(s)\mid k\in K, s\in I\}$
is a cohomogeneity one Lagrangian submanifold in $\mathcal{M}$.
Moreover,
if $\tau$ is a solution of the following ODE $(\ref{eqn:AIIIspecial})$,
then $L$ is a special Lagrangian submanifold with phase $\psi$$:$
\begin{equation}\label{eqn:AIIIspecial}
\mathrm{Im}\left(e^{i\psi}i^{2m-1}\tau'(\tan\tau)^{2m-3}(1+\tan^{2}\tau)^{3}(\cos 2\tau)^{2}\right)=0.
\end{equation}
\end{Thm}
\begin{proof}
Set $B=\{(\zeta,\xi)\in\mathbb{C}^{2m}\times\mathbb{C}^{2m}\mid
\zeta\neq0,\xi\cdot\bar{\zeta}=0\}$
and 
$\zeta=
\begin{pmatrix}
0 & A_{\theta} \\
-{}^{t}A_{\theta} & 0
\end{pmatrix}$ for $\theta\not\in(\pi/4)\mathbb{Z}$. 
If $(\zeta,\xi)\in B$ holds,
then $\xi\in\mathbb{C}^{2m}$ has the form:
\begin{equation*}
\xi=
\begin{pmatrix}
0 & X \\
-{}^{t}\bar{X} & 0
\end{pmatrix},\quad
X=\left(\begin{array}{ccccc}
\alpha\sin\theta & \xi_{12} & \xi_{13}  & \dotsb & \xi_{1m}\\
\xi_{21} & -\alpha\cos\theta & \xi_{23} & \dotsb & \xi_{2m}
\end{array}\right),\alpha,\xi_{ij}\in\mathbb{C}.
\end{equation*}
Let $(z;w)=\Phi(\zeta,\xi)$.
In this case, (\ref{eqn:alphastz}) yields
\begin{align*}
\alpha_{Stz}(Y^{*}_{(z;w)}&)=0\quad(\forall Y\in\mathfrak{k})\\
&\Longleftrightarrow
\alpha_{Stz}(X_{l})^{*}_{(z;w)}=0,
\alpha_{Stz}(Z_{ij})^{*}_{(z;w)}=0,
\alpha_{Stz}(W_{ij})^{*}_{(z;w)}=0,
\alpha_{Stz}(W_{2})^{*}_{(z;w)}=0\\
&\Longleftrightarrow
\xi_{ij}=0,\alpha=\lambda\in\mathbb{R},
\end{align*}
where we have used the condition
$\theta\not\in(\pi/4)\mathbb{Z}$.
Let $\tau=\tau(s)$ be a regular curve
in $\{z\in\mathbb{C} \mid 0 < |\mathrm{Re}\,z|<\pi/4\}(\subset \mathbb{C})$
and $\sigma=\sigma(s)$ be the curve in $\mathcal{M}$
defined by
\begin{equation*}
\sigma=(\left(\begin{array}{cc}
0 & X\\
-{}^{t}\bar{X}&0
\end{array}\right);\left(\begin{array}{cc}
0 & X\\
-{}^{t}\bar{X}&0
\end{array}\right))\quad
X=\left(\begin{array}{ccccc}
\cos\tau & 0 & 0 & \dotsb & 0\\
0 & \sin\tau & 0 & \dotsb & 0
\end{array}\right).
\end{equation*}
The curve $\sigma$ is contained in $\hat{\mu}^{-1}(0)$.
Hence
$L:=\{\hat{\rho}(k)\sigma(s)\mid k\in K, s\in I\}$ gives a
cohomogeneity one
Lagrangian submanifold in $\mathcal{M}$
(because of Proposition \ref{pro:HSpro2}).
With the use of above, we obtain
\begin{equation*}
\begin{split}
\Omega_{Stz}(\sigma', (X_{1})^{*}_{\sigma}, (X_{2})^{*}_{\sigma},
(Z_{12})^{*}_{\sigma},\dotsc,(Z_{2m})^{*}_{\sigma},
&(W_{12})^{*}_{\sigma},\dotsc,(W_{2m})^{*}_{\sigma},
(W_{2})^{*}_{\sigma})\\
&=2^{2m-1}i^{2m-1}\tau'(\tan\tau)^{2m-3}(1+\tan^{2}\tau)^{3}(\cos 2\tau)^{2}.
\end{split}
\end{equation*}
It follows from Proposition \ref{pro:HLpro}
that $L$ is a special Lagrangian submanifold (with phase $\psi$)
if $\tau$ is a solution of (\ref{eqn:AIIIspecial}).
\end{proof}

\subsection{Case of $(\mathfrak{u},\mathfrak{k})=(\mathfrak{o}(m+2),\mathfrak{o}(m)\oplus\mathfrak{o}(2))$}\label{subsec:BDI}

In this subsection,
we adopt the following notations (\cite[p.~453]{Hel}):
\begin{itemize}
\item $\mathfrak{u}=\mathfrak{o}(m+2)$
\item $\mathfrak{k}=\left\{\begin{pmatrix}
X_{1} & 0 \\
0 & X_{2}
\end{pmatrix}\in\mathfrak{u}\biggm| X_{1}\in\mathfrak{o}(2), X_{2}\in\mathfrak{o}(m)\right\}$
\item $\mathfrak{p}=\left\{\begin{pmatrix}
0 & X \\
-{}^{t}X & 0
\end{pmatrix}\in\mathfrak{u}\biggm| X : \text{$(2\times m)$-real matrix}\right\}$
\item $\mathfrak{a}=\left\{\begin{pmatrix}
0 & A \\
-{}^{t}A & 0
\end{pmatrix}\in\mathfrak{p}\biggm| A=\begin{pmatrix}
a_{1} & 0 & 0 & \dotsb & 0\\
0 & a_{2} & 0 & \dotsb & 0
\end{pmatrix}, a_{1},a_{2}\in\mathbb{R}\right\}$
\end{itemize}
In this setting,
$\mathfrak{a}$ is a maximal abelian subspace of $\mathfrak{p}$.
The following element $J_{0}$ in the center of $\mathfrak{k}$
gives the complex structure $\mathrm{ad}(J_{0})$ on $\mathfrak{p}$:
\begin{equation*}
J_{0}=\left(\begin{array}{cc|ccc}
0 & -1 & & &\\
1 &  0 & & &\\\hline
  &    & & &\\
  &    & \multicolumn{3}{c}{0}\\
  &    & & &
\end{array}\right).
\end{equation*}
Therefore we identify
$(\mathfrak{p},\mathrm{ad}(J_{0}))$ with $\mathbb{C}^{m}$
as follows:
\begin{equation*}
\left(\begin{array}{cc|ccc}
0       & 0       & x_{11}  & \ldots & x_{1m}\\
0       & 0       & x_{21}  & \ldots & x_{2m}\\\hline
-x_{11} & -x_{21} &         &        & \\
\vdots  & \vdots  &         & 0       & \\
-x_{1m} & -x_{2m} &         &        & \\
\end{array}\right)\in\mathfrak{p}
\longleftrightarrow
(x_{11}+ix_{21},\ldots,x_{1m}+ix_{2m})\in\mathbb{C}^{m}.
\end{equation*}
The sphere
$S^{2m-1}$ can be regarded as a subset of $\mathbb{C}^{m}$:
\begin{equation*}
S^{2m-1}=\{(z_{1},\ldots,z_{m})\in\mathbb{C}^{m}\mid |z_{1}|^{2}+\dotsb+|z_{m}|^{2}=1\}.
\end{equation*}
The above identification 
leads to
$\mathfrak{a}\cap S^{2m-1}=\{(\cos\theta,i\sin\theta,0,\ldots,0)\in\mathbb{C}^{m}\mid \theta\in\mathbb{R}\}$.
A basis of $\mathfrak{k}$ is given by $J_{0}$ and
\begin{equation*}
Y_{ij}=
\left(\begin{array}{cc|ccc}
0&0&&&\\
0&0&&&\\\hline
&&&&\\
&&&-E_{ij}^{(m)}+E_{ji}^{(m)}&\\
&&&&
\end{array}\right),\quad 1\leq i< j \leq m.
\end{equation*}
Therefore we can show that
the $K$-orbit through $\tilde{A}=\pi(A)\in \mathbb{C}P^{m-1}$ with
$A=(\cos\theta,i\sin\theta,0,\ldots,0)$
has codimension one if and only if $\theta\not\in(\pi/4)\mathbb{Z}$.
The tangent space at $\tilde{A}$ is spanned by
$(Y_{ij})^{*}_{\tilde{A}}$ for $i=1,2, i<j\leq m$,
where $\pi:\mathbb{C}^{m}-\{0\}\to\mathbb{C}P^{m-1}$ denotes the canonical projection.
In this subsection we prove the following result.

\begin{Thm}\label{thm:slagBDI}
Let $\tau:I \to \{z\in\mathbb{C}\mid 0<|\mathrm{Re}\,z|<\pi/4\}(\subset \mathbb{C})$ be a regular curve
and $\sigma:I \to \mathcal{M}(\cong\Phi(T\mathbb{C}P^{m-1}))$
be the curve defined by
\begin{equation*}
\sigma(s)=(\cos\tau(s),\sin\tau(s),0,\ldots,0;\cos\tau(s),i\sin\tau(s),0,\ldots,0).
\end{equation*}
Then $L:=\{\hat{\rho}(k)\sigma(s)\mid k\in K,s\in I\}$
is a cohomogeneity one Lagrangian submanifold in $\mathcal{M}$.
Moreover,
if $\tau$ is a solution of the following ODE $(\ref{eqn:BDIspecial})$,
then $L$ is a special Lagrangian submanifold with phase $\psi$$:$
\begin{equation}\label{eqn:BDIspecial}
\mathrm{Im}\left(e^{i\psi}i^{m-2}\tau'(\sin2\tau(s))^{m-3}\sin4\tau(s)\right)
=0.
\end{equation}
\end{Thm}
\begin{proof}
Set $B=\{(\zeta,\xi)\in\mathbb{C}^{m}\times\mathbb{C}^{m}\mid \zeta\neq0,\xi\cdot\bar{\zeta}=0\}$ and $\zeta=(\cos\theta,i\sin\theta,0,\ldots,0)\in\mathfrak{a}\cap S^{2m-1}$ for $\theta\not\in(\pi/4)\mathbb{Z}$.
If $(\zeta,\xi)\in B$ holds,
then $\xi=(\alpha i\sin\theta,\alpha\cos\theta,\xi_{3},\ldots,\xi_{m})$
for $\xi_{i}, \alpha\in\mathbb{C}$.
Let $(z;w)=\Phi(\zeta,\xi)$.
By using (\ref{eqn:alphastz}), we can verify
that
\begin{equation*}
\alpha_{Stz}(X^{*}_{(z;w)})=0\,(\forall X\in\mathfrak{k})
\Longleftrightarrow
\alpha_{Stz}(Y_{ij})^{*}_{(z;w)}=0
\Longleftrightarrow
\xi_{i}=0,\alpha=i\lambda\in i\mathbb{R},
\end{equation*}
where we have used the condition
$\theta\not\in(\pi/4)\mathbb{Z}$.
Let $\tau=\tau(s)$
be a regular curve
in $\{z\in\mathbb{C}\mid 0< |\mathrm{Re}\,z|<\pi/4\}(\subset \mathbb{C})$
and $\sigma=\sigma(s)$ be the curve in $\mathcal{M}$
defined by
\begin{equation*}
\sigma=
(\cos\tau,i\sin\tau,0,\ldots,0;\cos\tau,-i\sin\tau,0,\ldots,0).
\end{equation*}
The curve of $\sigma$ is contained in $\hat{\mu}^{-1}(0)$.
Hence
$L:=\{\hat{\rho}(k)\sigma(s)\mid k\in K, s\in I\}$
gives a Lagrangian submanifold in $\mathcal{M}$
(because of Proposition \ref{pro:HSpro2}).
From the above results, we obtain
\begin{equation*}
\Omega_{Stz}(\sigma',
(Y_{12})^{*}_{\sigma},\dotsc,(Y_{2m})^{*}_{\sigma})
=i^{m-1}\tau'\sin^{m-3}(2\tau)\sin(4\tau).
\end{equation*}
Hence it follows from Proposition \ref{pro:HLpro}
that, for a solution $\tau$ of (\ref{eqn:BDIspecial}),
$L$ is a special Lagrangian submanifold (with phase $\psi$).
\end{proof}

\subsection{Case of $(\mathfrak{u},\mathfrak{k})=(\mathfrak{o}(10),\mathfrak{u}(5))$}\label{subsec:DIII}

In this subsection,
we adopt the following notations (\cite[p.~453]{Hel}):
\begin{itemize}
\item $\mathfrak{u}=\mathfrak{o}(10)$
\item $\mathfrak{k}=\left\{\begin{pmatrix}
X_{1} & X_{2} \\
-X_{2} & X_{1}
\end{pmatrix}\in\mathfrak{u}\biggm|
\begin{tabular}{l}
$X_{1}\in\mathfrak{o}(5)$,\\
$X_{2}$: symmetric real $(5\times 5)$-matrix
\end{tabular}
\right\}$
\item $\mathfrak{p}=\left\{\begin{pmatrix}
X_{1} & X_{2} \\
X_{2} & -X_{1}
\end{pmatrix}\in\mathfrak{u}\biggm| X_{1}, X_{2}\in \mathfrak{o}(5)\right\}$
\item $\mathfrak{a}$ is the set
of all elements which have the form
$\begin{pmatrix}
A & 0\\
0 & -A
\end{pmatrix}\in\mathfrak{p}$ with
\begin{equation*}
A=\left(\begin{array}{cc|cc|c}
0&a_{1}&0&0&0\\
-a_{1}&0&0&0&0\\\hline
0&0&0&a_{2}&0\\
0&0&-a_{2}&0&0\\\hline
0&0&0&0&0
\end{array}\right),\quad a_{1},a_{2}\in \mathbb{R}.
\end{equation*}
\end{itemize}
In this setting, $\mathfrak{a}$ is a maximal abelian subspace of $\mathfrak{p}$.
The following element $J_{0}$ in the center of $\mathfrak{k}$
gives the complex structure $\mathrm{ad}(J_{0})$ on $\mathfrak{p}$:
\begin{equation*}
J_{0}=\dfrac{1}{2}\left(\begin{array}{c|c}
0 & E_{5}\\\hline
-E_{5} & 0
\end{array}\right).
\end{equation*}
Hence we identify $(\mathfrak{p},\mathrm{ad}(J_{0}))$ with $\mathbb{C}^{10}$
as follows:
\begin{equation*}
\left(\begin{array}{cc}
Y & X\\
X & -Y
\end{array}\right)
\in\mathfrak{p}
\longleftrightarrow
\left(\begin{array}{cc}
0 & X+iY\\
X+iY & 0
\end{array}\right)
\in\mathbb{C}^{10}.
\end{equation*}
The sphere
$S^{19}$ can be regarded as a subset of $\mathbb{C}^{10}$:
\begin{equation*}
S^{19}=\left\{
\left(\begin{array}{cc}
0 & Z \\
Z & 0
\end{array}\right)
\biggm|
\begin{array}{l}
Z=(z_{ij})\in\mathfrak{o}(5,\mathbb{C}),\\
|z_{12}|^{2}+|z_{13}|^{2}+\dotsb+|z_{45}|^{2}=1
\end{array}
\right\}.
\end{equation*}
Under the above identification, it is shown that
$\mathfrak{a}\cap S^{19}$ consists of matrices of the form
$\begin{pmatrix}
0 & A_{\theta} \\
A_{\theta} & 0
\end{pmatrix}$
with 
\begin{equation*}
A_{\theta}=\left(\begin{array}{cc|cc|c}
0&i\cos\theta&0&0&0\\
-i\cos\theta&0&0&0&0\\\hline
0&0&0&i\sin\theta&0\\
0&0&-i\sin\theta&0&0\\\hline
0&0&0&0&0
\end{array}\right),\quad \theta\in \mathbb{R}.
\end{equation*}
We consider elements $\tilde{Z}_{ij},\tilde{W}_{ij}\,(1\leq i < j \leq 5)$
and $\tilde{W}_{k}\,(1\leq k \leq 4)$
in $\mathfrak{k}$
defined as follows:
\begin{align*}
\tilde{Z}_{ij} &=\left(\begin{array}{cc}
-E_{ij}^{(5)}+E_{ji}^{(5)} & 0\\
0& -E_{ij}^{(5)}+E_{ji}^{(5)}
\end{array}\right),\\
\tilde{W}_{ij} &=\left(\begin{array}{cc}
0&E_{ij}^{(5)}+E_{ji}^{(5)}\\
-(E_{ij}^{(5)}+E_{ji}^{(5)})&0
\end{array}\right),\\
\tilde{W}_{k} &=\left(\begin{array}{cc}
0&E_{kk}^{(5)}-E_{k+1,k+1}^{(5)}\\
-(E_{kk}^{(5)}-E_{k+1,k+1}^{(5)})&0\\
\end{array}\right).
\end{align*}
Then $\tilde{Z}_{ij},\tilde{W}_{ij}\,(1\leq i < j \leq 5)$,
$\tilde{W}_{k}\,(1\leq k \leq 4)$ and $J_{0}$
give a basis of $\mathfrak{k}$.
We can verify
that
the $K$-orbit through $\tilde{A}=\pi(A)\in\mathbb{C}P^{9}$
with
$A=\begin{pmatrix}
0 & A_{\theta} \\
A_{\theta} & 0
\end{pmatrix}$
has codimension one if and only if $\theta\not\in(\pi/4)\mathbb{Z}$.
The tangent space at $\tilde{A}$
is spanned by
\begin{equation*}
\begin{cases}
(\tilde{Z}_{ij})^{*}_{\tilde{A}},\,(\tilde{W}_{ij})^{*}_{\tilde{A}}&
(i,j)=(1,3),(1,4),(1,5),(2,3),(2,4),(2,5),(3,5),(4,5),\\
(\tilde{W}_{2})^{*}_{\tilde{A}},
\end{cases}
\end{equation*}
where $\pi:\mathbb{C}^{10}-\{0\}\to\mathbb{C}P^{9}$
denotes the canonical projection.
In this subsection, we prove the following result.

\begin{Thm}\label{thm:slagDIII}
Let $\tau:I\to\{z\in\mathbb{Z}\mid 0 < |\mathrm{Re}\,z|<\pi/4\}(\subset \mathbb{C})$ be a regular curve and $\sigma:I\to\mathcal{M}(\cong\Phi(T\mathbb{C}P^{9}))$
be the curve defined by
$\sigma(s)=(\begin{pmatrix}
0 & Z \\
Z & 0
\end{pmatrix};\begin{pmatrix}
0 & W \\
W & 0
\end{pmatrix})$, where
\begin{align*}
Z&=\left(\begin{array}{cc|cc|c}
0&i\cos\tau(s)&0&0&0\\
-i\cos\tau(s)&0&0&0&0\\\hline
0&0&0&i\sin\tau(s)&0\\
0&0&-i\sin\tau(s)&0&0\\\hline
0&0&0&0&0
\end{array}\right),\\
W&=\left(\begin{array}{cc|cc|c}
0&-i\cos\tau(s)&0&0&0\\
i\cos\tau(s)&0&0&0&0\\\hline
0&0&0&-i\sin\tau(s)&0\\
0&0&i\sin\tau(s)&0&0\\\hline
0&0&0&0&0
\end{array}\right).
\end{align*}
Then $L:=\{\hat{\rho}(k)\sigma(s)\mid k\in K, s\in I\}$ is a cohomogeneity one Lagrangian submanifold in $\mathcal{M}$.
Moreover, if $\tau$ is a solution of the following ODE
$(\ref{eqn:DIIIspecial})$,
then $L$ is a special Lagrangian submanifold with phase $\psi$$:$
\begin{equation}\label{eqn:DIIIspecial}
\mathrm{Im}(e^{i\psi}i\tau'(1-\tan^{2}\tau)^{4}(1+\tan^{2}\tau)(\tan\tau)^{5})=0.
\end{equation}
\end{Thm}
\begin{proof}
Set $B=\{(\zeta,\xi)\in\mathbb{C}^{10}\times\mathbb{C}^{10}\mid\zeta\neq0,\xi\cdot\bar{\zeta}=0\}$ and $\zeta=\begin{pmatrix}
0 & A_{\theta} \\
A_{\theta} & 0
\end{pmatrix}$ for $\theta\not\in(\pi/4)\mathbb{Z}$.
If $(\zeta,\xi)\in B$ holds,
then
\begin{equation*}
\xi=\left(\begin{array}{cc}
0 & X\\
X & 0
\end{array}\right),\quad
X=\left(\begin{array}{cc|cc|c}
0&\alpha\sin\theta&\xi_{13}&\xi_{14}&\xi_{15}\\
-\alpha\sin\theta&0&\xi_{23}&\xi_{24}&\xi_{25}\\\hline
-\xi_{13}&-\xi_{23}&0&-\alpha\cos\theta&\xi_{35}\\
-\xi_{14}&-\xi_{24}&\alpha\cos\theta&0&\xi_{45}\\\hline
-\xi_{15}&-\xi_{25}&-\xi_{35}&-\xi_{45}&0
\end{array}\right)\in\mathfrak{o}(5,\mathbb{C}),\alpha,\xi_{ij}\in\mathbb{C}.
\end{equation*}
Let $(z;w)=\Phi(\zeta,\xi)$.
With the use of
 (\ref{eqn:alphastz}), we have
\begin{align*}
\alpha_{Stz}(X^{*}_{(z;w)})=0\,(\forall X\in\mathfrak{k})&\Longleftrightarrow
\alpha_{Stz}(\tilde{Z}_{ij})^{*}_{(z;w)}=0,\,
\alpha_{Stz}(\tilde{W}_{ij})^{*}_{(z;w)}=0,\,
\alpha_{Stz}(\tilde{W}_{2})^{*}_{(z;w)}=0\\
&\Longleftrightarrow
\xi_{ij}=0,\,
\alpha=i\lambda\in i\mathbb{R},
\end{align*}
where we have used the condition
$\theta\not\in(\pi/4)\mathbb{Z}$.
Let $\tau=\tau(s)$ be a regular curve
in $\{z\in\mathbb{C}\mid 0<|\mathrm{Re}\,z|<\pi/4\}(\subset\mathbb{C})$
and $\sigma=\sigma(s)$ be the curve in $\mathcal{M}$ defined by
$\sigma(s)=(\begin{pmatrix}
0 & Z \\
Z & 0
\end{pmatrix};\begin{pmatrix}
0 & W \\
W & 0
\end{pmatrix})$ with
\begin{align*}
Z&=\left(\begin{array}{cc|cc|c}
0&i\cos\tau(s)&0&0&0\\
-i\cos\tau(s)&0&0&0&0\\\hline
0&0&0&i\sin\tau(s)&0\\
0&0&-i\sin\tau(s)&0&0\\\hline
0&0&0&0&0
\end{array}\right),\\
W&=\left(\begin{array}{cc|cc|c}
0&-i\cos\tau(s)&0&0&0\\
i\cos\tau(s)&0&0&0&0\\\hline
0&0&0&-i\sin\tau(s)&0\\
0&0&i\sin\tau(s)&0&0\\\hline
0&0&0&0&0
\end{array}\right).
\end{align*}
The curve $\sigma$ is contained in $\hat{\mu}^{-1}(0)$.
Hence
$L:=\{\hat{\rho}(k)\sigma(s)\mid k\in K, s\in I\}$
gives a cohomogeneity one Lagrangian submanifold in
$\mathcal{M}$ (because of Proposition \ref{pro:HSpro2}).
From the above results, we obtain
\begin{equation*}
\Omega_{Stz}(\sigma',(\tilde{Z}_{13})^{*}_{\sigma}, \dotsc,
(\tilde{Z}_{45})^{*}_{\sigma},
(\tilde{W}_{13})^{*}_{\sigma}, \dotsc,
(\tilde{W}_{45})^{*}_{\sigma},
(\tilde{W}_{2})^{*}_{\sigma})=
2^{10}i(1-\tan^{2}\tau)^{4}(1+\tan^{2}\tau)(\tan\tau)^{5}.
\end{equation*}
Hence
it follows from Proposition \ref{pro:HLpro}
that,
for a solution $\tau$ of (\ref{eqn:DIIIspecial}),
$L$ is a special Lagrangian submanifold
(with phase $\psi$).
\end{proof}

\subsection*{Future direction}
We have constructed
cohomogeneity one special Lagrangian submanifolds
in $T^{*}\mathbb{C}P^{n}$,
which are obtained by
the linear isotropy actions of
the classical Hermitian symmetric spaces as in Table \ref{table:HSS}.
The rest of the work is the case when $(\mathfrak{u},\mathfrak{k})=(\mathfrak{e}_{6},\mathfrak{o}(10)\oplus\mathfrak{o}(2))$.
Though it is worth of applying the same method to the exceptional case, there is another interesting method.
Hashimoto-Mashimo (\cite{HM}) proposed
another systematic method to construct cohomogeneity one
special Lagrangian submanifolds in the cotangent bundle
over the sphere,
which is based on the geometry of restricted root system.
We will develop their method to the case of $T^{*}\mathbb{C}P^{n}$
and construct a cohomogeneity one special Lagrangian submanifold
for the linear isotropy action of
$(\mathfrak{u},\mathfrak{k})=(\mathfrak{e}_{6},\mathfrak{o}(10)\oplus\mathfrak{o}(2))$.

\end{document}